\documentclass{article}

\usepackage{amsmath,amsthm,amssymb,amsfonts}
\usepackage{verbatim}

\usepackage{stmaryrd} 
\usepackage{hyperref}

\usepackage[colorinlistoftodos]{todonotes}

\usepackage{tikz}
\usepackage{tkz-berge, tkz-graph}
\tikzset{vertex/.style={circle,draw,fill,inner sep=0pt,minimum size=1mm}}

\theoremstyle{plain}
\newtheorem{thm}{Theorem}
\newtheorem{lem}[thm]{Lemma}

\theoremstyle{definition}
\newtheorem{definition}[thm]{Definition}
\newtheorem{exl}[thm]{Example}

\numberwithin{thm}{section}

\def\Z{{\mathbb Z}}

\def\R{{\mathbb R}}

\begin{document}
\title{Shy Maps in Topology}
\author{Laurence Boxer
         \thanks{
    Department of Computer and Information Sciences,
    Niagara University,
    Niagara University, NY 14109, USA;
    and Department of Computer Science and Engineering,
    State University of New York at Buffalo.
    E-mail: boxer@niagara.edu
    }
}
\date{ }
\maketitle

\begin{abstract}
There is a concept in digital topology of a {\em shy map}. We define an analogous concept for 
topological spaces: We say a function is shy if it is continuous and the inverse image of
every path-connected subset of its image is path-connected. Some basic
properties of such maps are presented. For example, every shy map onto
a semilocally simply connected space induces a surjection of fundamental groups (but a shy
map onto a space that is not semilocally simply connected need not do so).

\end{abstract}

\section{Introduction}
Shy maps between digital images were introduced in~\cite{Bx05} and studied in subsequent papers belonging
to the field of digital topology,
including~\cite{Bx14,BxSt16,Bx-Shy,Bx-normal,Bx-alternate}. In this paper, we develop an analogous notion of a shy map between topological spaces and study its properties.

Recall 
that if
$F: X \to Y$ is a continuous function of topological spaces such that $F(x_0)=y_0$,
then $F$ induces a homomorphism of fundamental groups,
$F_*: \Pi_1(X,x_0) \to \Pi_1(Y,y_0)$, defined by
$F_*([f]) = [F \circ f]$ for every loop $f: (S^1,s_0) \to (X,x_0)$, where $S^1$ is
the unit circle in the Euclidean plane.

A topological space $X$ is {\em semilocally simply connected}~\cite{Brown}
if for every $x \in X$ there is
a neighborhood $N_x$ of $x$ in $X$ such that every loop in $N_x$ is nullhomotopic in $X$.

We let $\Z$ denote the set of integers, and $\R$, the real line.

A digital image is often considered as a graph $(X,\kappa)$, where $X \subset \Z^n$ for some
positive integer $n$ and
$\kappa$ is an adjacency relation on $X$. A function $f: (X, \kappa) \to (Y, \lambda)$
between digital images is {\em continuous} if for every $\kappa$-connected
subset $A$ of $X$, $f(A)$ is a $\lambda$-connected subset of $Y$~\cite{Rosenfeld87,Bx99}.
A continuous surjection $f: (X, \kappa) \to (Y, \lambda)$ between digital images is called {\em shy}~\cite{Bx05} if the following hold.
\begin{itemize}
\item For all $y \in Y$, $f^{-1}(y)$ is a $\kappa$-connected subset of $X$, and
\item for all pairs of $\lambda$-adjacent $y_0,y_1 \in Y$, $f^{-1}(\{y_0,y_1\})$ is a
      $\kappa$-connected subset of $X$.
\end{itemize}
It is shown in~\cite{BxSt16} that a continuous function between digital images
is shy if and only if for every
$\lambda$-connected subset $Y'$ of $Y$, $f^{-1}(Y')$ is a $\kappa$-connected subset of $X$.
Since connectedness for a graph is analogous to what 
topologists call path-connectedness, we use the following.

\begin{definition}
\label{shy-def}
Let $X$ and $Y$ be topological spaces and let $f: X \to Y$. Then
$f$ is {\em shy} if $f$ is continuous and for every path-connected $Y' \subset f(X)$, 
$f^{-1}(Y')$ is a path-connected subset of $X$. $\Box$
\end{definition}

Note a shy map in digital topology is defined to be a surjection~\cite{Bx05}, which we do not require here.
This makes the requirement $Y' \subset f(X)$ of Definition~\ref{shy-def} noteworthy. For example, the
embedding $f: [0,\pi] \to S^1$ given by $f(x) = (\cos x, \sin x)$ is shy according to Definition~\ref{shy-def}. It would not be shy were Definition~\ref{shy-def} written with the requirement
$Y' \subset f(X)$ replaced by the requirement $Y' \subset Y$, since, e.g., the arc $A$ of $S^1$ from $(-1,0)$
to $(1,0)$ containing $(0,-1)$ is path connected, but $f^{-1}(A) = \{0,\pi\}$ is not.

Some authors consider digital images as topological spaces rather than as graphs, applying
the {\em Khalimsky topology} to digital images (see, e.g.,~\cite{Herman,KhalimskyEtal,Kong03,KongKopperman, Kopperman, Melin}).
The Khalimsky topology on $\Z$ takes a basic neighborhood of an integer $z$ to be $\{z\}$ if $z$ is odd;
$\{z-1,z,z+1\}$ if $z$ is even. The quotient map $q: \R \to \Z$ given by
\[ q(x) = \left \{ \begin{array}{ll}
         x & \mbox{if } x \mbox{ is an even integer;} \\
         y & \mbox{if } y \mbox{ is the unique odd integer such that } |x - y| < 1,
\end{array}
\right .
\]
is easily seen to be a shy map. More generally, the quotient map $q^n$ that is the $n$-fold
product of $q$ as a map from $\R^n$ (with the Euclidean topology) to $\Z^n$ (with the
Cartesian product topology taken from the Khalimsky topology on $\Z$) is a shy map.
The shyness of this quotient map is among the reasons why $\R^n$ is a useful ``continuous analog"
of $\Z^n$ and $\Z^n$ is an interesting ``discrete analog" of $\R^n$.

We mention here that the term {\em path in} $X$ from $x_0$ to $x_1$ will be used in two senses, as
is common practice: It may mean a continuous function $f: [a, b] \to X$ such that $f(a) = x_0$
and $f(b) = x_1$, or may mean the image $f([a, b])$ of such a function.

\section{Induced surjection on fundamental group}
In~\cite{Bx05}, it was shown that a digital shy map induces a surjection of digital fundamental
groups. In this section, we derive an analogous result for shy maps into 
semilocally simply connected spaces.

In this section, we let $e: [0,1] \to S^1$ be defined by 
$e(t) = (\cos 2\pi t, \sin 2\pi t)$.

\begin{lem}
({\rm \cite{H&W}}, Exercise 4, p. 269)
\label{extensionLemma}
Let $f,g: [0,1] \to X$ be paths into the topological space $X$, such that $f(0)=g(0)$ and
$f(1)=g(1)$. Let $h: S^1 \to X$ be defined by
\[ h(e(t)) = \left \{ \begin{array}{ll}
   g(2t) & \mbox{if } 0 \le t \le 1/2; \\
   f(2-2t) & \mbox{if } 1/2 \le t \le 1.
\end{array} \right .
\]
Then $f \simeq g$ via a homotopy that holds the endpoints fixed if and only if there is a
continuous extension of $h$ to the interior of $S^1$. $\Box$
\end{lem}

The following is an immediate consequence of Lemma~\ref{extensionLemma}.

\begin{lem}
\label{htpyHoldingEnds}
Let $\alpha$, $\beta: [0,1] \to Y$ be paths into a topological space $Y$ such
that $\alpha(0)=\beta(0)$ and $\alpha(1)=\beta(1)$. Let $\overline{\alpha}$ be the reverse
path of $\alpha$. If $\overline{\alpha} \cdot \beta$ is nullhomotopic in $Y$, then there
is a homotopy between $\alpha$ and $\beta$ that holds the endpoints fixed. $\Box$
\end{lem}

The main result of this section is the following.

\begin{thm}
\label{inducedOnto}
Let $X$ and $Y$ be topological spaces
such that $Y$ is semilocally simply connected,
and let $F: X \to Y$ be a shy surjection, with $F(x_0)=y_0$. Then the induced homomorphism
$F_*: \Pi_1(X,x_0) \to \Pi_1(Y,y_0)$, defined by
$F_*([f])= [F \circ f]$, is onto.
\end{thm}

\begin{proof}
Let $g: (S^1,s_0) \to (Y, y_0)$ be a pointed loop in $Y$.
It suffices to show there is a pointed loop $f: (S^1,s_0) \to (X,x_0)$
in $X$ such that $F \circ f$ is pointed homotopic to $g$.

For any integer $n \ge 2$, we construct a loop $g_n: S^1 \to Y$ as follows. Without loss
of generality, we assume the base point $s_0$ of $S^1$ is $e(0)$. Note that
$x_0 \in F^{-1}(g(s_0))$. For $i \in \{1,\ldots, n-1 \}$ let $s_i = e(i/n)$ and let
$x_i \in F^{-1}(g(s_i))$. For $i \in \{0,\ldots, n-1 \}$, 
let $A_i = e([\frac{i}{n}, \frac{i+1}{n}])$ (so $A_i$ is an arc of $S^1$ from $s_i$ to
$s_{(i+1) \mod n}$).

Since $F$ is shy, $F^{-1}(g(A_i))$ is path-connected. Moreover, 
$\{s_i, s_{(i+1) \mod n}\} \subset A_i$. Therefore, there exists a
continuous $f_i: A_i \to F^{-1}(g(A_i))$ such that $f_i(s_i)=x_i$ and
$f_i(s_{(i+1) \mod n})=x_{(i+1) \mod n}$. Let $f: S^1 \to X$ be the
loop such that $f|_{A_i}=f_i$ for all $i \in \{0,\ldots,n-1\}$. Let $g_n = F \circ f$.

Since $g(S^1)$ is compact and $Y$ is semilocally simply connected, there is a finite list 
$U_1, \ldots, U_k$ of open sets in $Y$ that cover $g(S^1)$ such that each loop into any $U_j$
is nullhomotopic in $Y$. On applying Lebesgue's covering lemma to the open cover
$\{g^{-1}(U_1), \ldots, g^{-1}(U_k)\}$ of $S^1$, we see that if $n$ is sufficiently large then
for each $i \in \{0, \ldots, n-1\}$ the set $g(A_i)=g(e([\frac{i}{n}, \frac{i+1}{n}]))$ is contained in one of the sets
$U_1,\ldots, U_k$.
Therefore, we can apply Lemma~\ref{htpyHoldingEnds} to conclude that
for sufficiently large $n$, $i \in \{0,\ldots,n-1\}$, and any homeomorphism
$h_i: [0,1] \to A_i$, the paths $g_n|_{A_i}$ and $g|_{A_i}$ are homotopic in $Y$ relative to $\{0,1\}$,
whence $g_n|_{A_i}$ and $g|_{A_i}$ are homotopic in
$Y$ relative to $\{h_i(0), h_i(1)\} = \{s_i, s_{(i+1) \mod n}\}$. It follows that, for sufficiently large $n$,
$g$ and $g_n= F \circ f$ are homotopic in $Y$ relative to $\{s_0\}$.
This establishes the assertion.
\end{proof}

Theorem~\ref{inducedOnto} may fail if $Y$ is not semilocally simply connected,
as shown by the following example.

\begin{exl}
Let $H$ be a Hawaiian earring 
in $\R^2$, let $a$ be the common point of the circles of $H$, and let $H'$ be the reflection of $H$ in the
line through $a$ that is tangent to the circles of $H$ (so that $H \cap H' = \{a\}$).
Let $Y = H \cup H'$, let
\[X=(H \times \{0\}) \cup (H' \times \{1\}) \cup (\{a\} \times [0, 1]) \subset \R^3, \]
and let $p: \R^3 \to \R^2$ be the
projection map $p(x, y, z) = (x, y)$. Then $p|_X: X \to Y$ is a shy surjection that
does not induce a surjection of fundamental groups: If $C_n$ and $C_n'$
respectively denote the $n^{th}$ largest circles of $H$ and $H'$, then there is no loop
$\ell: S^1 \to X$ such that $p|_X \circ \ell$ is homotopic
in $Y$ to a loop that winds around all the circles of $Y$ in the order
$C_1, C_1', C_2, C_2', C_3, C_3', \ldots$. $\Box$
\end{exl}

\section{Operations that preserve shyness}
It is shown in~\cite{Bx-Shy} that a composition of
shy surjections between digital images is shy. The following
gives an analogous result for shy maps between
topological spaces.

\begin{thm}
Let $f: X \to Y$ and $g: Y \to Z$ be shy maps between
topological spaces. Suppose $f$ is a surjection.  Then $g \circ f: X \to Z$ is shy.
\end{thm}

\begin{proof}
It is clear that $g \circ f$ is continuous.
Let $A \subset g(Y)$ be path-connected. Since $g$ is shy,
$g^{-1}(A)$ is path-connected. Since $f$ is a shy surjection,
$f^{-1}(g^{-1}(A))$ $= (g \circ f)^{-1}(A)$ is path-connected.
Therefore, $g \circ f$ is shy.
\end{proof}

The following is suggested by analogous results for shy
maps between digital
images~\cite{Bx-Shy,Bx-normal,Bx-alternate}.

\begin{thm}
\label{factor}
Let $f_i: X_i \to Y_i$ be functions between topological
spaces, $1 \le i \le v$. Let
$f = \Pi_{i=1}^v f_i: \Pi_{i=1}^v X_i \to \Pi_{i=1}^v Y_i$
be the product function,
\[ f(x_1,\ldots,x_v) = (f_1(x_1), \ldots, f_v(x_v))
\mbox{ for } x_i \in X_i.
\]
If $f$ is shy, then each $f_i$ is shy.
\end{thm}

\begin{proof}
We use the well known continuity of projection maps
$p_j: \Pi_{i=1}^v Y_i \to Y_i$, defined by
$p_j(y_1, \ldots, y_v) = y_j$.

For some $x_i \in X_i$, let $I_j: X_j \to X=\Pi_{i=1}^v X_i$ be the injection defined by
\[I_j(x) = 
(x_1, \ldots, x_{j-1}, x, x_{j+1}, \ldots, x_v).\]

Since $f$ is shy, $f$ must be continuous,
and it follows that each $f_i=p_i \circ f \circ I_i$ is continuous.

Let $A_i$ be a path-connected subset of $f_i(X_i)$. We must
show that $f_i^{-1}(A_i)$ is a path-connected subset of
$X_i$. Fix $y_j \in A_j$ for all indices $j$ and let 
$I_j': Y_j \to Y=\Pi_{i=1}^v Y_i$ be the injection defined by
\[I_j'(y) = 
(y_1, \ldots, y_{j-1}, y, y_{j+1}, \ldots, y_v).\]
Then $I_i'(A_i)$ is path-connected, since it is
homeomorphic to $A_i$. Therefore,
$f^{-1}(I_i'(A_i)) = \Pi_{j=1}^v f_j^{-1}(p_j^Y(I_i'(A_i)))$, where
$p_j^Y: Y \to Y_j$ is the projection to the $j^{th}$ coordinate, is
path-connected.

Then $f_i^{-1}(A_i) = p_i(\Pi_{j=1}^v f_j^{-1}(p_j^Y(I_i'(A_i))))$ is path-connected.
This completes the proof.
\end{proof}

The proofs in~\cite{Bx-Shy,Bx-normal,Bx-alternate} for
analogs of the converse of Theorem~\ref{factor} rely
on structure that digital images have but that
cannot be assumed for topological spaces. It
appears that obtaining either a proof or a counterexample
for the converse of Theorem~\ref{factor} is a difficult problem.

We say a topological space $W$ is the {\em wedge} of its
subsets $X$ and $Y$, denoted by $W = X \vee Y$, if $W = X \cup Y$ and
$X \cap Y = \{x_0\}$ for some $x_0 \in W$.
We say subsets $A,B$ of a topological space $X$ are {\em separated in} $X$ if neither $A$ nor
$B$ intersects the closure in $X$ of the other. We have the following, which has an
elementary proof that is left to the reader.

\begin{lem}
\label{baseptSeparates}
Suppose $W = A \vee B$, with $A \setminus B$ and $B \setminus A$ separated in $W$. Let $C$ be
a path-connected subset of $W$ such that $C \cap A \ne \emptyset \ne C \cap B$. Then the
unique point of $A \cap B$ is in $C$. Further, each of $C \cap A$ and $C \cap B$ is
path-connected. $\Box$
\end{lem}

We will use the following ``Gluing Rule.''

\begin{thm}
\label{gluing}
{\rm \cite{Brown}}
Let $f: X \to Y$ be a function between topological spaces. Suppose $X = A \cup B$ where
$A \setminus B \subset Int A$, $B \setminus A \subset Int B$. If $f|_A$ and $f|_B$ are
continuous, then $f$ is continuous. $\Box$
\end{thm}

If $W = X \vee Y$ with $X \cap Y = \{x_0\}$, $W' = X' \vee Y'$ with $X' \cap Y' = \{x_0'\}$, and
$f: (X,x_0) \to (X', x_0')$ and
$g: (Y,x_0) \to (Y', x_0')$ are pointed functions, let
$f \vee g: W \to W'$ be the function defined by
\[ (f \vee g)(a) = \left \{ \begin{array}{ll}
    f(a) & \mbox{if } a \in X; \\
    g(a) & \mbox{if } a \in Y.
\end{array} \right .
\]
The following is suggested by an analogous result for
digital images~\cite{Bx-Shy}.

\begin{thm}
Let $W = X \vee Y$, $W' = X' \vee Y'$,
$f: (X, x_0) \to (X', x_0')$, $g: (Y,x_0) \to (Y', x_0')$, and $h = f \vee g: W \to W'$,
where $x_0$ and $x_0'$ are the unique points of of $X \cap Y$ and $X' \cap Y'$, respectively.
Assume $X \setminus Y$ and $Y \setminus X$ are separated in $W$, and
$X' \setminus Y'$ and $Y' \setminus X'$ are separated in $W'$
Then $h$ is shy if and only if $f$ and $g$ are both shy.
\end{thm}

\begin{proof}
As $f = h|_X$ and $g = h|_Y$, and the sets $X \setminus Y$ and $Y \setminus X$ are separated,
it follows from Theorem~\ref{gluing} that
$h = f \vee g$ is continuous if and only if each of $f$ and $g$ is continuous.

To see that $f$ is shy if $h$ is shy we note that, if $P'$ is a path-connected subset of $f(X)$,
then $h^{-1}(P')$ is path-connected and so $h^{-1}(P') \cap X = f^{-1}(P')$ is path-connected by
Lemma~\ref{baseptSeparates}. Similarly, $g$ is shy if $h$ is shy.

Suppose each of $f$ and $g$ is shy. To prove $h$  is shy, let $Q'$ be
any path-connected subspace of $h(X \vee Y)$. We want to show $h^{-1}(Q')$ is path-connected.
We see from
Lemma~\ref{baseptSeparates} that if $Q' \cap X' \ne \emptyset \ne Q' \cap Y'$, then
$x_0' \in Q'$ and, moreover, $Q' \cap X'$ and $Q' \cap Y'$ are path-connected subsets of $f(X)$
and $g(Y)$, respectively. This and the shyness of $f$ and $g$ imply that each of the sets
$f^{-1}(Q' \cap X')$ and $g^{-1}(Q' \cap Y')$ is path-connected, and that $x_0$ lies in both sets if both sets are nonempty. So the union of the two sets is
also path-connected. As the union is $h^{-1}(Q')$, the proof is complete.
\end{proof}

\section{Shy maps into $\R$}
In this section, we show that shy maps into the
reals have special properties. Our results are suggested by 
analogs for digital images~\cite{Bx-Shy}.

\begin{thm}
Let $X$ be a connected subset of $\R$ and
let $f: X \to \R$ be continuous. Then
$f$ is shy if and only if $f$ is monotone non-decreasing or monotone non-increasing.
\end{thm}

\begin{proof}
Suppose $f$ is shy. If $f$ is not monotone, then
there exist $a,b,c \in X$ such that $a < b < c$ and either
\begin{itemize}
\item $f(a) < f(b)$ and $f(b) > f(c)$, or
\item $f(a) > f(b)$ and $f(b) < f(c)$.
\end{itemize}
In the former case, the continuity of $f$
implies there exist $a' \in [a,b)$ and $c' \in (b,c]$ such that
$\{a',c'\} \subset f^{-1}(\{\max\{f(a),f(c)\}\})$ but
$b \not \in f^{-1}(\{\max\{f(a),f(c)\}\})$. Therefore,
$f^{-1}(\{\max\{f(a),f(c)\}\})$ is not path-connected,
contrary to the assumption that $f$ is shy. The
latter case generates a contradiction similarly.
We conclude that $f$ is monotone.

Suppose $f$ is a monotone function. Let $A$ be a
path-connected subset of $f(X)$. Let $c,d \in f^{-1}(A)$.
Without loss of generality, $c \le d$ and $f(c) \le f(d)$.
Since $f$ is continuous and $X$ is connected, for every
$x$ such that $c \le x \le d$ we have 
$f(c) \le f(x) \le f(d)$.
Therefore, $f^{-1}(A)$ contains
$[c, d]$, a path from $c$ to $d$.
Thus, $f$ is shy.
\end{proof}

\begin{thm}
Let $f: S^1 \to \R$ be shy.
Then $f$ is a constant function.
\end{thm}

\begin{proof}
Suppose the image of $f$ has distinct points $a,b$, where $a < b$. Then
there exist $a' \in f^{-1}(a)$, $b' \in f^{-1}(b)$.
There are two distinct arcs $A_0$ and $A_1$ in $S^1$ from
$a'$ to $b'$ such that $S^1=A_0 \cup A_1$ and
$A_0 \cap A_1 = \{a',b'\}$. Since $f$ is continuous,
there exist $c_0 \in A_0$, $c_1 \in A_1$ such that
$f(c_0)=f(c_1) = (a+b)/2$. Then
$f^{-1}(\{(a+b)/2\})$ is disconnected, contrary to the
assumption that $f$ is shy. Therefore, $f$ is a constant function.
\end{proof}

\begin{thm}
Let $X$ be a path-connected topological space for which
there exists $r \in X$ such that $X \setminus \{r\}$ is
not path-connected. Let $f: X \to \R$ be a shy map. Then there are at most 2 path
components of $X \setminus \{r\}$ on which $f$ is
not identically equal to the constant function with value $f(r)$.
\end{thm}

\begin{proof}
Suppose there exist 2 distinct path components $A$, $B$ of
$X \setminus \{r\}$ on which $f$ is not the constant
function with value $f(r)$. Then there exist
$a \in A$, $b \in B$ such that $f(a) \ne f(r) \ne f(b)$.

First, we show that either $f(a) < f(r) < f(b)$ or
$f(a) > f(r) > f(b)$. Suppose otherwise.
\begin{itemize}
\item Suppose $f(a) < f(r)$ and $f(b) < f(r)$. 
      Without loss of generality, $f(a) \le f(b) < f(r)$.
      Since $X$ is pathwise connected and $A$ is a {\em maximal} path-connected subset of
      $X \setminus \{r\}$,
      it is easy to see that there exists a
      path $P$ in $A \cup \{r\}$ from $a$ to $r$.
      By continuity of $f$, it is easy to see that there exists $x_0 \in P$
      such that $f(x_0)=f(b)$. But then
      $x_0$ and $b$ are in distinct path-components of
      $f^{-1}(\{f(b)\})$, contrary to the assumption that
      $f$ is shy.
\item If $f(a) > f(r)$ and $f(b) > f(r)$, then we      
      similarly obtain a contradiction.
\end{itemize}
We conclude that either $f(a) < f(r) < f(b)$ or
$f(a) > f(r) > f(b)$.

So if there is a 3rd path-component $C$ of $X \setminus \{r\}$ on
which $f$ is not identically equal to $f(r)$,
then there exists $c \in C$ such that either
\begin{itemize}
\item $f(c) < f(r)$, in which case we get a contradiction
      as in the first case above, since both
      $f(c)$ and $\min\{f(a),f(b)\}$ are less than 
      $f(r)$; or
\item $f(c) > f(r)$, in which case both
      $f(c)$ and $\max\{f(a),f(b)\}$ are greater than 
      $f(r)$, so we have a contradiction as in the second
      case above.
\end{itemize}
Therefore, there cannot be a 3rd path-component $C$ of 
$X \setminus \{r\}$ on which $f$ is not identically equal
to $f(r)$.
\end{proof}

\section{Concluding remarks}
Drawing on the notion of a shy map between digital
images~\cite{Bx05},
we have introduced an analogous notion of a shy map
between topological spaces. We have shown that shy maps
between topological spaces have many properties analogous to 
those of shy maps between digital images.

\section{Acknowledgment}
The corrections and excellent suggestions of the anonymous reviewer are gratefully acknowledged.

\end{document}